\documentclass[bibother,reqno]{asl}
\usepackage{rotate}

\usepackage[UKenglish]{babel}
\usepackage[utf8]{inputenc}
\usepackage{verbatim}

\usepackage{theoremref}
\usepackage{hyperref}
\usepackage{graphicx}

\usepackage[normalem]{ulem}

\usepackage{dsfont}
\usepackage{enumerate}

\usepackage{thmtools}

\usepackage{catoptions}
\makeatletter

\def\Autoref#1{%
	\begingroup
	\edef\reserved@a{\cpttrimspaces{#1}}%
	\ifcsndefTF{r@#1}{%
		\xaftercsname{\expandafter\testreftype\@fourthoffive}
		{r@\reserved@a}.\\{#1}%
	}{%
		\ref{#1}%
	}%
	\endgroup
}
\def\testreftype#1.#2\\#3{%
	\ifcsndefTF{#1autorefname}{%
		\def\reserved@a##1##2\@nil{%
			\uppercase{\def\ref@name{##1}}%
			\csn@edef{#1autorefname}{\ref@name##2}%
			\autoref{#3}%
		}%
		\reserved@a#1\@nil
	}{%
		\autoref{#3}%
	}%
}
\makeatother

\usepackage{nicefrac}

\theoremstyle{plain}
\newtheorem{theorem}{Theorem}[section]
\newtheorem{thmx}{Theorem}

\newtheorem{corollary}[theorem]{Corollary}
\newtheorem{lemma}[theorem]{Lem\-ma}
\newtheorem{proposition}[theorem]{Prop\-o\-si\-tion}

\theoremstyle{definition}

\newtheorem{remark}[theorem]{Remark}
\newtheorem{example}[theorem]{Example}

\numberwithin{equation}{section}

\newcommand{\N}[0]{\mathbb{N}}
\newcommand{\Z}[0]{\mathbb{Z}}
\newcommand{\Q}[0]{\mathbb{Q}}
\newcommand{\R}[0]{\mathbb{R}}

\newcommand{\OO}[0]{\mathcal{O}}
\newcommand{\MM}[0]{\mathcal{M}}

\newcommand{\pow}[1]{((#1))}

\renewcommand{\L}[0]{\mathcal{L}}
\newcommand{\Lor}[0]{\mathcal{L}_{\mathrm{or}}}
\newcommand{\Log}[0]{\mathcal{L}_{\mathrm{og}}}
\newcommand{\Lr}[0]{\mathcal{L}_{\mathrm{r}}}
\newcommand{\Lvf}[0]{\mathcal{L}_{\mathrm{vf}}}
\newcommand{\Lovf}[0]{\mathcal{L}_{\mathrm{ovf}}}

\newcommand{\ol}[1]{\overline{#1}}
\newcommand{\ul}[1]{\underline{#1}}

\newcommand{\vnat}[0]{v_{\mathrm{nat}}}

\renewcommand{\sf}[0]{\mathrm{f}}

\newcommand{\LL}[0]{\mathcal{L}}

\newcommand{\KK}[0]{\mathcal{K}}

\DeclareMathOperator{\tp}{tp}

\DeclareMathOperator{\mf}{\mathbf{f}}
\DeclareMathOperator{\mr}{\mathbf{r}}
\DeclareMathOperator{\mv}{\mathbf{v}}

\usepackage{xcolor}

\def\confname{}
\def\copyrightyear{2000}

\title[Definable valuations on ordered fields]{Definable valuations on ordered fields}

\author[P.~Dittmann]{Philip Dittmann}
\revauthor{Dittmann, Philip}
\author[F.~Jahnke]{Franziska Jahnke}
\revauthor{Jahnke, Franziska}
\author[L.~S.~Krapp]{Lothar Sebastian Krapp}
\revauthor{Krapp, Lothar Sebastian}
\author[S.~Kuhlmann]{Salma Kuhlmann}
\revauthor{Kuhlmann, Salma}

\address{Institut für Algebra\\Fachrichtung Mathematik\\TU Dresden\\01062 Dresden, Germany}
\email{philip.dittmann@tu-dresden.de}
\urladdr{https://tu-dresden.de/mn/math/algebra/das-institut/beschaeftigte/philip-dittmann}

\address{Mathematisches Institut\\Fachbereich Mathematik und Informatik\\Universität Münster\\Einsteinstraße 62\\48149 Münster, Germany}
\email{franziska.jahnke@uni-muenster.de}
\urladdr{https://ivv5hpp.uni-muenster.de/u/fjahn\_01/}
\thanks{The second author was funded by the Deutsche Forschungsgemeinschaft (DFG, German Research Foundation) under Germany's Excellence Strategy EXC 2044-390685587, Mathematics M\"unster: Dynamics-Geometry-Structure, as well as by a Fellowship from the Daimler and Benz Foundation.}

\address{Fachbereich Mathematik und Statistik\\Universität Konstanz\\78457 Konstanz, Germany}
\email{sebastian.krapp@uni-konstanz.de}
\urladdr{http://www.math.uni-konstanz.de/\urltilde krapp/}
\thanks{The third author was partially supported by Werner und Erika Messmer Stiftung.}

\address{Fachbereich Mathematik und Statistik\\Universität Konstanz\\78457 Konstanz, Germany}
\email{salma.kuhlmann@uni-konstanz.de}
\urladdr{https://www.mathematik.uni-konstanz.de/kuhlmann/}
\thanks{The fourth author wishes to acknowledge support from the AFF Universität Konstanz for the Project TRAGVAL: Topics in Real Algebraic Geometry and Valuation Theory.}

\thanks{Part of this work was carried out while the third and fourth authors were generously hosted by the Fields Institute, in Toronto, during the month of June 2022.}

\begin{document}
	
	\begin{abstract}
		We study the definability of convex valuations on ordered fields, with a particular focus on the distinguished subclass of henselian valuations. In the setting of ordered fields, one can consider definability both in the language of rings $\Lr$ and in the richer language of ordered rings $\Lor$. We analyse and compare definability in both languages and show the following contrary results: while there are \emph{convex} valuations that are definable in the language $\Lor$ but not in the language $\Lr$, any $\Lor$-definable \emph{henselian} valuation  is already $\Lr$-definable. To prove the latter, we show that the value group and the \emph{ordered} residue field of an ordered henselian valued field are stably embedded (as an ordered abelian group, respectively as
		an ordered field). Moreover, we show that in almost real closed fields \emph{any} $\Lor$-definable valuation is  henselian.
	\end{abstract}
	
	\maketitle
	
	\section{Introduction}
	
	One of the main objectives in the model-theoretic study of fields is the analysis of first-order definable\footnote{Throughout this work, `definable' always means `definable \emph{with} parameters'.} sets and substructures. 
	Given a field, it is a natural question to ask whether a given valuation ring is a definable subset in some expansion of the language $\Lr=\{+,-,\cdot, 0, 1\}$ of rings.
	A key reason to study definability of valuation rings 
	is to transfer questions of decidability and existential decidability (i.e., the
	question whether Hilbert's Tenth Problem 
	has a positive solution)
	between different rings and fields. 
	However, there is also a more recent motivation stemming from
	classifying fields within Shelah's classification hierarchy:
	whereas stable (or, more generally, simple) fields do not admit
	any non-trivial 
	$\Lr$-definable 
	valuations, a conjecture going back to Shelah
	predicts
	that infinite NIP fields which are neither real closed nor separably 
	closed admit a non-trivial $\Lr$-definable henselian valuation.
	In a recent series of spectacular papers, this was shown to hold
	in the `finite-dimensional' (i.e., dp-finite) case by Johnson
	\cite{johnson}.
	For a survey on definability of henselian valuations, mostly 
	in the language of rings, see \cite{fehmjahnkesurvey}.
	
	In this work, we primarily study valuations on \emph{ordered} fields.
	This allows us to also consider their definability in the language of \emph{ordered} rings $\Lor = \Lr\cup\{<\}$. We focus on convex valuations, i.e., valuations whose valuation ring is convex with respect to the ordering, as these naturally induce an ordering on the residue field (cf.\ \cite[Proposition~2.2.4]{englerprestel}). Note that due to 
	\cite[Lemma~4.3.6]{englerprestel}, 
	every henselian valuation on an ordered field is already convex.
	
	By considering the expanded language $\Lor$ rather than $\Lr$, one may expect further definability results. Indeed, we present examples of ordered fields with \emph{convex} valuations that are $\Lor$-definable but not $\Lr$-definable (see \Autoref{ex:nonlrdef1} and \Autoref{ex:nonlrdef2}). Rather surprisingly, for \emph{henselian} valuations the language $\Lor$ does not produce any further definability results, that is, every $\Lor$-definable henselian valuation is already $\Lr$-definable (see \Autoref{thm:main}). In the particular case of almost real closed fields, $\Lor$-definability even suffices to ensure both henselianity (thus convexity) and $\Lr$-definability (see \Autoref{thm:arcmain}).
	
	The structure of this paper is as follows.
	After introducing preliminary notions and results in \Autoref{sec:prelims}, we first turn to the definability of convex valuations in \Autoref{sec:convex}. We establish conditions on the value group and the residue field ensuring the definability (\emph{with} and \emph{without} parameters) of a given convex valuation (see \Autoref{thm:convexmain} and \Autoref{cor:convexmain}).
	Subsequently we compare these results to other known definability conditions in the literature (see \Autoref{rmk:comparison}) and construct our main examples \Autoref{ex:nonlrdef1} and \Autoref{ex:nonlrdef2} to show that there are convex valuations that are $\emptyset$-$\Lor$-definable but not $\Lr$-definable.
	Lastly, we answer \cite[Question~7.1]{krappkuhlmannlehericyforum} positively by presenting in  \Autoref{ex:density} an ordered valued field that is dense in its real closure but still admits a non-trivial $\emptyset$-$\Lor$-definable convex valuation.
	In \Autoref{sec:qe} we turn to ordered \emph{henselian} valued fields and establish in \Autoref{thm:stablemb} that their value group (as an ordered abelian group) and their residue field (as an ordered field) are always stably embedded as well as orthogonal. As a result, we obtain that within ordered fields, every $\Lor$-definable coarsening of an $\Lr$-definable henselian valuation is already $\Lr$-definable (see \Autoref{cor:coarseninglr}) and that any $\Lor$-definable (not necessarily convex) valuation is comparable to any henselian valuation (see \Autoref{prop:ord-hens-comp}). The special class of almost real closed fields, which are ordered fields admitting a henselian valuation with real closed residue field, is studied in \Autoref{sec:arc}. We show in \Autoref{thm:arcmain} that within almost real closed fields any $\Lor$-definable valuation (which \emph{a priori} does not have to be convex) is already henselian and $\Lr$-definable, thereby giving a negative answer to \cite[Question~7.3]{krappkuhlmannlehericyforum}. Building on the results of the previous sections, we finally prove in \Autoref{sec:henselian} the main theorem of this paper stating as follows:\footnote{This result will be restated as \Autoref{thm:main}.}

	\begin{thmx}[Main Theorem]
		Let $(K,<)$ be an ordered field and let $v$ be a henselian valution on $K$. If $v$ is $\Lor$-definable, then it is already $\Lr$-definable.
	\end{thmx}
	
	\section{Preliminaries}\label{sec:prelims}
	
	We denote by $\N$ the set of natural numbers \emph{without} $0$ and by $\omega$ the set of natural numbers \emph{with} $0$.
	
	We mostly follow valuation-theoretic notation of \cite{englerprestel}.
	For a valuation $v$ on a field $K$, we write $\OO_v$ for its valuation ring, $\MM_v$ for the maximal ideal of $\OO_v$, $Kv = \OO_v/\MM_v$ for the residue field and $vK$ for its value group (written additively).
	For an element $x \in \OO_v$, its residue $x + \MM_v \in Kv$ is denoted $\overline{x}$, where the valuation $v$ in question will always be clear from context.
	
	Given an ordering $<$ on $K$ (always compatible with the field structure), a valuation $v$ is called \emph{convex} (with respect to $<$) if $\OO_v$ is a convex set in the usual sense.
	See \cite[Section 2.2.2]{englerprestel} for a number of equivalent conditions.
	In particular, $v$ is convex if and only if the ordering $<$ induces an ordering on the residue field $Kv$ in the natural way.
	This residue ordering will then also be denoted by $<$, which should not lead to confusion.
	
	For a given field $K$ and ordered abelian group $G$, we write $K((G))$ for the \emph{Hahn field} consisting of those formal sums $\sum_{g \in G} a_g t^g$ with coefficients $a_g \in K$, $t$ a formal variable, whose support is well-ordered.
	See for instance \cite[Section 3.1]{dries} for details.
	We will generally endow $K((G))$ with its natural valuation with value group $G$, assigning to each element of $K((G))^\times$ the order of the lowest non-zero coefficient.
	An ordering $<$ on $K$ can naturally be extended to $K((G))$ by stipulating that an element of $K((G))^\times$ is positive if and only if its non-zero coefficient of lowest order is positive.
	We will denote this ordering on $K((G))$ again by $<$.
	
	For background on the model theory of valued fields see \cite{dries}, or \cite{marker} for model theory more generally.
	We consider fields as structures in the language of rings $\Lr = \{ +, -, \cdot, 0, 1\}$, ordered abelian groups as structures in the language $\Log = \{ +, -, 0, < \}$ in the natural way.
	Given an ordering $<$ on a field $K$, we consider the ordered field $(K,<)$ as a structure in the language $\Lor = \Lr \cup \{ < \}$.
	Given a valuation $v$ on a field $K$, we usually work in the one-sorted language $\Lvf = \Lr \cup \{ \OO \}$, where the unary predicate $\OO$ is to be interpreted as the valuation ring $\OO_v \subseteq K$.
	In Section \ref{sec:qe}, we will also work in a three-sorted language.
	For a field $K$ with an ordering $<$ and a valuation $v$, we will use the language $\Lovf = \Lvf \cup \{ < \}$ for $(K, <, v)$.

	If $a$ is an element of an ordered abelian group (respectively, ordered field), we denote its absolute value $\max\{a,-a\}$ by $|a|$.
	
	A set is \emph{$\LL$-definable} if it is definable in the first-order language $\LL$. If we wish to specify that the parameters can be chosen to come from a specific set $A$, we write \emph{$A$-$\LL$-definable}.
	
	\section{Convex Valuations}\label{sec:convex}
	
	We start by giving sufficient conditions on the value group or the residue field of a convex valuation $v$ such that $v$ is $\Lor$-definable. By this, we strengthen all of the three cases given in \cite[Theorem~5.3]{krappkuhlmannlehericyforum}. Subsequently, we present several cases in which the given valuation is already $\Lor$-definable without parameters and discuss how these cases generalise other known definability results in the literature.
	
	\begin{theorem}\label{thm:convexmain}
		Let $(K,<)$ be an ordered field and let $v$ be a convex valuation on $K$. Suppose that at least one of the following holds:
		\begin{enumerate}[(i)]
			\item\label{thm:convexmain:1} $vK$ is discretely ordered, i.e., admits a least positive element.
			
			\item\label{thm:convexmain:2} $vK$ is not closed in its divisible hull.
			
			\item\label{thm:convexmain:3} $Kv$ is not closed in its real closure.
		\end{enumerate}
		Then $v$ is $\Lor$-definable. Moreover, in the cases \eqref{thm:convexmain:1} and \eqref{thm:convexmain:2}, $v$ is definable by a formula using only one parameter.
	\end{theorem}
	\begin{proof}
		We may assume that $v$ is non-trivial.
		\begin{enumerate}[(i)]
			\item Since $vK$ is discretely ordered, we can choose $b \in K^\times$ such that $v(b)$ is the minimal positive element of $vK$.
			Note that for every $x\in \MM_v$, we have $v(x^2/b) = 2v(x)-v(b)> 0$.
			Since every element $y \in \MM_v$ satisfies $|y|<1$, 
			we deduce that $\MM_v = \{ x \in K \mid |x^2/b| < 1 \}$. Hence, $\MM_v$ is $\{ b \}$-$\Lor$-definable, and $\OO_v$ can be defined in terms of $\MM_v$.
			
			\item Since $vK$ is not closed in its divisible hull, we can take $\gamma \in vK$ and $n > 1$ such that $\gamma \not\in n \cdot vK$ but every open interval in $vK$ containing $\gamma$ contains an element of $n \cdot vK$.
			Let $b \in K$ with $b > 0$ and $v(b) = \gamma$, and set
			\begin{equation}\label{eq:dense-div-hull-1}
			S_b := \{ x \in K  \mid  x\geq 0\text{ and } x^n/b < 1 \} = \{ x \in K  \mid  x\geq 0\text{ and } n v(x) > \gamma \}
			\end{equation}
			(where the equality uses that we cannot have $v(x^n/b) = 0$ since $\gamma \notin  n \cdot vK$).
			It now suffices to prove that
			\begin{equation}\label{eq:dense-div-hull-2}
			\OO_v = \{ y \in K \mid y^4 S_b \subseteq S_b \},
			\end{equation}
			since the set on the right-hand side is $\{ b \}$-$\Lor$-definable.
			The inclusion $\subseteq$ is clear since the condition $n v(x) > \gamma$ in \eqref{eq:dense-div-hull-1} is stable under multiplying $x$ by an element of $\OO_v$.
			
			For the inclusion $\supseteq$, suppose that $y \in K \setminus \OO_v$, so $v(y) < 0$.
			By the choice of $\gamma$, we can take $z \in K$ with $z>0$ and $\gamma + v(y) < n v(z) < \gamma -v(y)$.
			Now $z/y^2 \in S_b$ since $n v(z/y^2) = n v(z) - 2n v(y) > n v(z) - v(y) > \gamma$, but $y^4 (z/y^2) \not\in S_b$ since $n v(y^4 (z/y^2)) = n v(z) + 2n v(y) < n v(z) + v(y) < \gamma$.
			This proves that $S_b$ is not stable under multiplication by $y^4$, completing the proof of \eqref{eq:dense-div-hull-2}.
			
			\item			
			Let $f\in Kv[X]$ be the minimal polynomial of an element $x_0\in R\setminus Kv$, where $R$ denotes the real closure of $(Kv,<)$, such that $x_0$ can be arbitrarily approximated by elements of $Kv$. Then there are $a,b\in Kv$ with $a<x_0<b$ such that the following hold:
			\begin{enumerate}
				\item 
				The polynomial $f$ has exactly one zero in $\{x\in R\mid a\leq x\leq b\}$. In particular, $f$ changes sign precisely once in this interval.
				\item 
				For any $\varepsilon\in Kv$ with $0<\varepsilon<b-a$, there exists $x\in Kv$ with $a<x<x+\varepsilon < b$ such that $f(x)f(x+\varepsilon) < 0$.
			\end{enumerate}
			Passing to $-f$ if necessary, we may assume that $f(a) < 0 < f(b)$.  
			Let $F \in K[X]$ be a lift of $f$, let $a_0, b_0 \in K$ be lifts of $a$ and $b$, and consider the $\Lor$-definable set $S$ given by
			\[\{ x \in K \mid  a_0 \leq x \leq b_0 \text{ and } F(x) < 0 \} = \{ x \in K \mid  a_0 \leq x \leq b_0 \text{ and } f(\ol{x}) < 0 \}, \]
			where the equality uses that $f$ has no zero in $Kv$.
			
			It now suffices to prove that for any $y\in K$ we have
			\[ y + S \subseteq S \iff y \in \MM_v\text{ and } y \geq 0, \]
			since then $\MM_v$  and therefore $\OO_v$ are $\Lor$-definable.
			For the implication $\Leftarrow$, let $y \in \MM_v$ be non-negative and $x \in S$.
			Then we have $x + y \geq x \geq a_0$ and $f(\overline{x+y}) = f(\overline x) < 0$. In particular $\overline{x+y} < b$ and thus $x+y < b_0$.
			Hence $y+x \in S$, as desired.
			
			For the implication $\Rightarrow$, let $y\in K$ with $y + S \subseteq S$.
			Since $a_0 \in S$, we must have $a_0 + y \in S$. Thus, $a_0 \leq a_0 + y \leq b_0$, implying that $0\leq y\leq b_0-a_0$ and $y \in \OO_v$.
			Note that $\ol{y}\neq b-a$, as otherwise $f(\ol{y}+a)=f(b)>0$, contradicting the fact that $y+a_0\in S$. Hence, $\ol{y}<b-a$.
			
			In order to show that $y \in \MM_v$, suppose for a contradiction that $v(y) = 0$, so $0<\overline{y}<b-a$.
			By choice of $f$, we can find $z \in \OO_v$ with $a < \overline z < b$ and $f(\overline z) < 0 < f(\overline{z+y})$.
			Then we have $z \in S$ but $z + y \not\in S$, in contradiction to our assumption $y + S \subseteq S$.
		\end{enumerate}
	\end{proof}
	
	In the following corollary, we point out several distinguished cases in which we obtain $\Lor$-definability without parameters.
	
	\begin{corollary}\label{cor:convexmain}
		Let $(K,<)$ be an ordered field and let $v$ be a convex valuation on $K$. Suppose that at least one of the following holds:
		\begin{enumerate}[(i)]
			\item\label{cor:convexmain:1} $vK$ is $p$-regular but not $p$-divisible for some prime $p\in \N$.\footnote{Equivalently, $vK$ contains a rank $1$ convex subgroup $H$ that is not $p$-divisible but $vK/H$ is $p$-divisible. 
				See \cite[Section~2.2]{hongthesis} for further characterisations of $p$-regular ordered abelian groups.}
			
			\item\label{cor:convexmain:3} $Kv$ is dense in its real closure but not real closed.
			
		\end{enumerate}
		Then $v$ is $\emptyset$-$\Lor$-definable.
	\end{corollary}
	
	\begin{proof}
		In both cases, at least one of the three conditions in \Autoref{thm:convexmain} is satisfied: if $vK$ is $p$-regular but not $p$-divisible, then it is either discrete or not closed in its divisible hull (cf.~\cite[Proposition~3.3]{krappkuhlmannlink}). 
		Thus, there exists an $\Lor$-formula $\psi(x,\ul{z})$ and a parameter tuple $\ul{b}\in K$ such that $\psi(x,\ul{b})$ defines $\OO_v$.
		\begin{enumerate}[(i)]
			\item\label{cor:convexmain:1b} 
			For any non-trivial convex subgroup $C\leq vK$ we have that $vK/C$ is $p$-divisible
			(cf.\ \cite[page~14]{hong}).
			Thus, 
			any strict coarsening of $v$ has a $p$-divisible value group. As $vK$ is not $p$-divisible,  $\OO_v$ is defined by the $\Lor$-formula $\varphi(y)$ expressing the following:
			\begin{align*}\exists \ul{z}\ &
				(\text{`}\psi(x,\ul{z}) \text{ defines a non-trivial convex valuation ring whose }\\
				&\quad \text{value group contains an element that is not $p$-divisible'} \wedge  \psi(y,\ul{z})).\end{align*}
			
			\item\label{cor:convexmain:3b} For any strict refinement $w$ of $v$ we have that $Kw$ is real closed. Indeed, since $Kv$ is dense in its real closure and the induced valuation $\overline{w}$ on $Kv$ is non-trivial and convex, we have that $Kw=(Kv)\overline{w}$ is real closed (cf.~\cite[Corollary~4.9]{krappkuhlmannlehericyforum}).
			Let $\theta$ be an $\Lor$-sentence that is true in the theory of real closed fields but does not hold in $Kv$.\footnote{For instance, $\theta$ may express that there exists a polynomial of a certain degree that does not have a zero.}			
			Then $\OO_v$ is defined by the $\Lor$-formula $\varphi(y)$ expressing the following:
			\begin{align*}\forall \ul{z}\ &
				(\text{`}\psi(x,\ul{z}) \text{ defines a non-trivial convex valuation ring whose }\\
				&\quad \text{residue field does not satisfy $\theta$'} \to  \psi(y,\ul{z})).\end{align*}
			
		\end{enumerate}
	\end{proof}
	
	\begin{remark}\thlabel{rmk:comparison}
		\begin{enumerate}[(i)]
			\item 		The cases \eqref{thm:convexmain:1} and \eqref{thm:convexmain:2} of \Autoref{thm:convexmain} are optimal in the sense that, in general, one cannot obtain parameter-free definability. More precisely, in \cite[Example~4.9~\&~Example~4.10]{krappkuhlmannlink} two ordered valued fields $(L_1,<,v_1)$ and $(L_2,<,v_2)$ are presented such that the following hold:
			\begin{itemize}
				\item $v_1$ and $v_2$ are henselian and thus convex;
				
				\item neither $v_1$ nor $v_2$ is $\emptyset$-$\Lor$-definable;
				
				\item $v_1L_1$ is discrete, and $v_2L_2$ is not closed in its divisible hull.
			\end{itemize}
			
			\item 
			The results for $\Lr$-definability of henselian (rather than convex) valuations  corresponding to
			\Autoref{thm:convexmain}~\eqref{thm:convexmain:1} and \eqref{thm:convexmain:2} as well as \Autoref{cor:convexmain}~\eqref{cor:convexmain:1} 
			are proven in \cite[Corollary~2]{hong}, \cite[Theorem~A]{krappkuhlmannlink} and \cite[Lemmas~2.3.6 and 2.3.7]{hongthesis} respectively.
			
			\item \Autoref{cor:convexmain} applies in particular if $vK$ is of rank $1$ (i.e., $v$ is the coarsest non-trivial convex valuation on $K$) but non-divisible, or if $Kv$ is archimedean (i.e., $v$ is the finest convex valuation) but not real closed.
		\end{enumerate}
	\end{remark}
	
	We now apply the $\Lor$-definability results above in order to obtain convex non-henselian valuations that are definable in the language $\Lor$ but not in the language $\Lr$.
	
	\begin{lemma}\thlabel{lem:undefinable}
		Let $K=\Q(s_i\mid i\in \omega)$, where $\{s_i\mid i\in \omega\}$ is algebraically independent over $\Q$. Suppose that $v$ is any valuation on $K$ with $v(s_i)\geq 0$ for any $i\in \N$ and $v(s_0)<0$. Then $v$ is not $\Lr$-definable in $K$.
	\end{lemma}
	
	\begin{proof}
		First note that $K$ is the $\Lr$-definable closure of $S:=\{s_i\mid i\in \omega\}$ in $K$. Hence, any $\Lr$-definable subset of $K$ is $S$-$\Lr$-definable.
		
		Assume, for a contradiction, that some $\Lr$-formula $\varphi(x,\ul{s})$ defines $\OO_v$, where $\ul{s}=(s_0,s_1,\ldots,s_n)$ for some $n\in \N$. Since also the set $\{s_0,s_1,\ldots,s_n,s_{n+1}+s_0,s_{n+2},\ldots\}$ is algebraically independent over $\Q$, we can set $\alpha$ to be the uniquely determined $\Lr$-automorphism on $K$ with
		$$\alpha(s_i):=\begin{cases}
		s_i&\text{ for } i\in \omega\setminus\{ n+1\},\\
		s_{n+1}+s_0&\text{ for } i=n+1.
		\end{cases}$$
		Now since $v(s_{n+1})\geq 0$, we have
		$$K\models \varphi(s_{n+1},\ul{s}).$$
		As $\alpha(s_{n+1})=s_{n+1}+s_0$ and $\alpha(\ul{s})=\ul{s}$, we obtain
		$$K\models \varphi(s_{n+1}+s_0,\ul{s})$$
		i.e., $s_{n+1}+s_0\in \OO_v$. However, $v(s_{n+1}+s_0)=v(s_0)<0$, a contradiction.
	\end{proof}
	
	\begin{example}\thlabel{ex:nonlrdef1}
		We construct an ordered valued field $(K,<,v)$ such $K$ is a subfield of the Laurent series field $\R\pow{\Z}$, $vK$ is discretely ordered, $Kv$ is archimedean and $v$ is $\emptyset$-$\Lor$-definable but not $\Lr$-definable.
		
		Let $k=\Q(s_1,s_2,\ldots)\subseteq \R$ for some set $\{s_i\mid i\in \N\}\subseteq \R$ that is algebraically independent over $\Q$.
		Consider the field $K = k(t)$, which we endow with the valuation $v$ and ordering $<$ given as the restriction of the valuation and ordering on the Hahn field $k((t)) = k((\Z))$.
		Then $vK=\Z$ and $Kv=k$, which is archimedean. \Autoref{cor:convexmain} shows that $v$ is $\emptyset$-$\Lor$-definable.
		Setting $s_0=t^{-1}$, we obtain $K=\Q(s_0,s_1,\ldots)$ and $v(s_0)=-1<0$ as well as $v(s_i)=0$ for any $i\in \N$. Hence, \Autoref{lem:undefinable} implies that $v$ is not $\Lr$-definable.
		\qed
	\end{example}
	
	\begin{example}\thlabel{ex:nonlrdef2}
		We construct an ordered valued field $(K,<,v)$ such that $K$ is a subfield of the Puiseux series field $\bigcup_{n\in \N}\R\pow{t^{1/n}} \subseteq \R\pow{\Q}$, $vK$ is densely ordered, $Kv$ is archimedean and $v$ is $\emptyset$-$\Lor$-definable but not $\Lr$-definable.	
		
		Let $\{r_i\mid i\in \N\}\subseteq \R$ be an algebraically independent set over $\Q$. We set $$s_0=t^{-1}\in \R\pow{\Q} \text{ and }s_i=r_it^{{1}/{i}}\in  \R\pow{\Q}$$
		for any $i\in \N$. Let 
		$K=\Q(s_0,s_1,\ldots) \subseteq \R\pow{\Q}$,
		and endow $K$ with the ordering $<$ and valuation $v$ given as the restriction of the Hahn field ordering and valuation on $\R\pow{\Q}$.
		The archimedean residue field of $K$ is not real closed, as $Kv\subseteq \Q(r_1,r_2,\ldots)$ and thus, for instance, $\sqrt{2}\notin Kv$.
		Hence,  \Autoref{cor:convexmain}~\eqref{cor:convexmain:3} shows that $v$ is $\emptyset$-$\Lor$-definable. Since $v(s_i)=\tfrac{1}{i}$ for any $i\in \N$, we have $vK=\Q$. Finally, we can apply \Autoref{lem:undefinable} to show that $v$ is not $\Lr$-definable, as $v(s_0)=-1<0$ and $v(s_i)=\tfrac 1i>0$ for any $i\in \N$.
		\qed
	\end{example}
	
	To complete this section, we relate the property of an ordered field to admit a non-trivial $\Lor$-definable convex valuation to the property of being dense in the real closure. It is known that an ordered field is either dense in its real closure or admits a non-trivial $\Lor$-definable convex valuation (see \cite[Proposition~6.5]{jahnkesimonwalsberg} and \cite[Fact~5.1]{krappkuhlmannlehericyforum}).
	However, the question whether these two cases are non-exclusive (see \cite[Question~7.1]{krappkuhlmannlehericyforum}) has so far been open.
	In the following, we answer this question positively by presenting an ordered field that is dense in its real closure and whose natural valuation is $\emptyset$-$\Lor$-definable.
	
	\begin{example}\thlabel{ex:density}
		Let $$K=\R\big( t^{1/n}\ \big|\ n \in \N\big)\subseteq \R\pow{\Q},$$
		and endow $K$ with the ordering $<$ and valuation $v$ given by restricting the ordering and the valuation of the Hahn field $\R\pow{\Q}$.
		Since $Kv =\R$ is real closed and $vK=\Q$ is divisible and of rank $1$, \cite[Proposition~9]{viswanathan} implies that $K$ is dense in its real closure.
		
		We first claim that the subset $\R \subseteq K$ is defined by the parameter-free $\Lr$-formula $\varphi(x)$ given by
		\[ \exists y \ 1 + x^4 = y^4 .\]
		Clearly $\R \subseteq \varphi(K)$, since for any $a \in \R$ we have $\sqrt[4]{1 + a^4} \in \R$.
		For the other inclusion, take $a, y \in K$ with $1 + a^4 = y^4$.
		There exists $N \in \N$ such that $a, y \in \R(t^{1/N})$.
		Letting $s = t^{1/N}$, we have $\R(s) \models \varphi(a)$, and therefore $a \in \R$ by \cite[Lemma 2]{malcev}.
		This shows $\R = \varphi(K)$, as desired.
		
		Therefore $\OO_v$, the convex hull of $\R$ in $K$, is defined by the parameter-free $\Lor$-formula
		\[ \exists x_1, x_2\ (\varphi(x_1) \wedge \varphi(x_2) \wedge x_1 \leq z \leq x_2). \]
		\qed
		
	\end{example}
	
	Note that \Autoref{ex:density} presents an ordered valued field with real closed residue field and divisible value group. Hence, none of the cases in \Autoref{cor:convexmain} can be applied, but we still obtain $\emptyset$-$\Lor$-definability of the valuation.
	
	\section{Stable Embeddedness}\label{sec:qe}

	In this section, we establish that 
	the ordered value group and the ordered residue field of an ordered henselian valued field are stably embedded and orthogonal (see \Autoref{thm:stablemb}). 
	Stable embeddedness and orthogonality are best-known to hold in the unordered situation of
	henselian valued fields of equicharacteristic $0$ 
	(cf., e.g.,  \cite[Section~5]{dries}), and  various
	other well-behaved settings (see also \cite[Section~8.3]{aschenbrenner} and \cite[page~171~f.]{jahnkesimon}). 
	As the key technical tool in the ordered context, we use Farr\'e's Embedding Lemma (\cite[Theorem~3.4]{farre}).	
	
	We consider the three-sorted language $\Lovf'$ given by
	$$\Lovf' = (\Lor,\Lor,\Log; \ol{\ \cdot\ }, v).$$
	The three sorts are denoted by $\mf$ (field sort), $\mr$ (residue field sort) and $\mv$ (value group sort). The two unary function symbols $\ol{\ \cdot\ }$ and $v$ have sorts $\ol{\ \cdot\ }\colon\mf\to\mr$ and $v\colon \mf\to\mv$.
	
	Let $(K,<,v)$ be an ordered valued field with convex valuation $v$.
	Then $(K,\allowbreak <,v)$ induces an $\Lovf'$-structure $$\mathcal{K}=((K,<),(Kv,<),vK;v,\ol{\ \cdot\ }),$$ where the domains of the valuation and the residue map are extended to $K$ by setting $v(0)=0$ and $\ol{a}=0$ for any $a\in K\setminus \OO_v$.
	When considering definability in $\mathcal{K}$, we allow parameters from all sorts as usual.
	An $\Lor$-formula $\varphi(\underline y)$ with parameters from $Kv$ may be considered as an $\Lovf'$-formula with parameters from $\KK$, where the variables $\underline y$ become $\mr$-variables.
	Similarly, an $\Log$-formula $\varphi(\underline z)$ with parameters from $vK$ may be considered as an $\Lovf'$-formula with parameters from $\mathcal{K}$, where the variables $\underline z$ become $\mv$-variables.
	
	In the following, we prove a weak version of relative quantifier-elimination; we only consider formulas whose variables are varying over residue field and value group.
	
	\begin{lemma}\thlabel{thm:qe}
		Let $(K,<,v)$ be an ordered henselian valued field and let $T$ be the diagram of the $\Lovf'$-structure $\KK$ as above, i.e., the complete theory of $\KK$ in the language $\Lovf'$ expanded by constants for all elements of $\KK$.
		Further, let $\ul{y}$ and $\ul{z}$ be tuples of distinct $\mr$- and $\mv$-variables, respectively.
		Then any $\Lovf'$-formula $\varphi(\ul{y},\ul{z})$ with parameters from $\KK$ is $T$-equivalent to an $\Lovf'$-formula of the form
		\begin{align}(\psi_1(\ul{y})\wedge \theta_1(\ul{z}))\vee\ldots \vee (\psi_N(\ul{y})\wedge \theta_N(\ul{z}))\label{eq:specialform}\end{align}
		for some $N\in \N$, where all $\psi_i$ are $\Lor$-formulas with parameters from $Kv$ and all $\theta_i$ are $\Log$-formulas with parameters from $vK$.
	\end{lemma}
	
	\begin{proof}
		Let $\Theta$ be the set of all $\Lovf'$-formulas of the form \eqref{eq:specialform}, i.e., of all finite disjunctions of conjunctions of an $\Lor$-formula and an $\Log$-formula with parameters from $Kv$ and $vK$, respectively.
		Modulo logical equivalence, $\Theta$ contains $\top$ as well as $\bot$ and is closed under finite disjunctions, finite conjunctions and negation. Hence by \cite[Corollary~B.9.3]{aschenbrenner}, we only have to verify the following:
		\emph{Let  $p$ and $q$ be any two complete $T$-realisable $(\ul{y},\ul{z})$-types with $p\cap \Theta = q\cap \Theta$. Then $p=q$.}
		
		Let $p$ and $q$ be as described above and let $\MM$ be a sufficiently saturated elementary extension of $\KK$ in which $p$ and $q$ are realised. Denote by $(M,<,v)$ the ordered henselian valued field inducing $\MM$, and let $\ul{r},\ul{r}^*\in Mv$ and $\ul{g},\ul{g}^*\in vM$ with $\MM\models p(\ul{r},\ul{g}) \wedge q(\ul{r}^*,\ul{g}^*)$. 
		By construction of the set $\Theta$, we have \begin{align}\tp^{vM}(\ul{g}/vK)=\tp^{vM}(\ul{g}^*/vK)\text{ and }\tp^{(Mv,<)}(\ul{r}/Kv)=\tp^{(Mv,<)}(\ul{r}^*/Kv).\label{eq:types}\end{align}

		Now let $\KK\preceq \MM_0\preceq \MM$, with $\MM_0$ smaller than the saturation of $\MM$, such that $\ul{r}\in M_0v$ and $\ul{g}\in vM_0$, where $(M_0,<,v)$ denotes the ordered henselian valued field inducing $\MM_0$. Due to \eqref{eq:types}, we can fix an $\Lor$-elementary embedding $\sigma\colon (M_0v,<)\to (Mv,<)$ over $Kv$ with $\sigma(\ul{r})=\ul{r}^*$ and an $\Log$-elementary embedding $\rho\colon vM_0\to vM$ over $vK$ with $\rho(\ul{g})=\ul{g}^*$ (cf.~\cite[Proposition~4.1.5]{marker}). 
		
		The quotient $vM_0/vK$ is torsion free, as $vK\preceq vM_0$. We can thus apply \cite[Theorem~3.4]{farre} (with all appearing levels equal to $1$) in order to obtain an embedding $\iota\colon (M_0,<,v)\to (M,<,v)$ over $K$ inducing both $\sigma$ and $\rho$. 
		Moreover, since both $\sigma$ and $\rho$ are elementary embeddings, \cite[Corollary 4.2~(ii)]{farre} implies that $(\iota(M_0),<,v)\preceq (M,<,v)$. Let $\MM_0'$ be the $\Lovf'$-structure induced by $(\iota(M_0),<,v)$ and denote by $h$ the isomorphism $(\iota,\sigma,\rho)\colon \MM_0\to \MM_0'$ over $\KK$.
		For any $\varphi(\ul{y},\ul{z})\in p$ we have $\MM_0\models \varphi(\ul{r},\ul{g})$. By applying $h$, we obtain $\MM_0'\models \varphi(\sigma(\ul{r}),\rho(\ul{g}))$ and hence $\MM\models \varphi(\ul{r}^*,\ul{g}^*)$. This establishes $p\subseteq q$.
		The other inclusion follows likewise.
	\end{proof}

	\begin{theorem}\thlabel{thm:stablemb}
		Let $(K,<,v)$ be an ordered henselian valued field inducing the $\Lovf'$-structure $\KK$. Then for any $m,n\in \omega$ the following hold:
		\begin{enumerate}[(i)]
			\item\label{thm:stablemb:1} Any subset of $(Kv)^m\times (vK)^n$ definable in $\KK$ is a finite union of rectangles of the form $Y\times Z$, where $Y\subseteq (Kv)^m$ is $\Lor$-definable in $(Kv,<)$ and $Z\subseteq (vK)^n$ is $\Log$-definable in $vK$.
			
			\item\label{thm:stablemb:2} For any set $B\subseteq \OO_v^m$ that is $\Lovf$-definable in $(K,<,v)$, the set $\ol{B}:=\{(\ol{b_1},\ldots,\ol{b_m})\mid (b_1,\ldots,b_m)\in B\}\subseteq (Kv)^m$ is $\Lor$-definable in $(Kv,<)$.
			
			\item\label{thm:stablemb:3} For any set $C\subseteq (K^\times)^n$ that is $\Lovf$-definable in $(K,<,v)$, the set $v(C):=\{(v(c_1),\ldots,v(c_n))\mid (c_1,\ldots,c_n)\in C\}\subseteq (vK)^n$ is $\Log$-definable in $vK$.
		\end{enumerate}
	\end{theorem}
	
	\begin{proof}
		By \Autoref{thm:qe}, any subset of $(Kv)^m\times (vK)^n$ definable in $\KK$ can be defined by a formula of the form \eqref{eq:specialform}. This immediately implies \eqref{thm:stablemb:1}. In order to obtain  \eqref{thm:stablemb:2} and \eqref{thm:stablemb:3}, it remains to notice that $\ol{B}\times v(C)\subseteq (Kv)^m\times (vK)^n$ is definable in $\KK$.
	\end{proof}

	\begin{corollary}\thlabel{cor:stablemb}
		Let $(K,<,v)$ be an ordered henselian valued field and let $w$ be an $\Lovf$-definable valuation on $K$. Then the following hold:
		\begin{enumerate}[(i)]
			\item \label{cor:stablemb:2} 
			If $w$ is a refinement of $v$, then the valuation $\ol{w}$ induced by $w$ on $Kv$ is $\Lor$-definable in $(Kv,<)$.		
			
			\item \label{cor:stablemb:1} 
			If $w$ is a coarsening of $v$, then $v(\OO_w^\times)$ is $\Log$-definable in $vK$.
		\end{enumerate}
	\end{corollary}
	
	\begin{proof}
		Both $\OO_w^\times$ and $\OO_w$ are $\Lovf$-definable in $K$. It remains to apply \Autoref{thm:stablemb}~\eqref{thm:stablemb:2} to $B=\OO_w\subseteq \OO_v$ in order to obtain \eqref{cor:stablemb:2} and \Autoref{thm:stablemb}~\eqref{thm:stablemb:3} to $C=\OO_w^\times\subseteq K^\times$ in order to obtain \eqref{cor:stablemb:1}.
	\end{proof}
	
	\begin{corollary}\thlabel{cor:coarseninglr}
		Let $(K,<)$ be an ordered field, let $v$ be an $\Lr$-definable hen\-se\-li\-an valuation on $K$ and let $w$ be an $\Lor$-definable coarsening of $v$. Then $w$ is already $\Lr$-definable.
	\end{corollary}
	
	\begin{proof}
		\Autoref{cor:stablemb}~\eqref{cor:stablemb:1} shows that $H=v(\OO_w^\times)$ is $\Log$-definable in $vK$. 
		Since $wK=vK/H$, for any $x\in K$ we have $x\in \OO_w$ if and only if $v(x)\geq 0 \vee v(x)\in H$.
		As $v$ is $\Lr$-definable in $K$, the latter can be expressed as an $\Lr$-formula with parameters from $K$.
	\end{proof}
	
	For later use, we also deduce the following.
	\begin{proposition}\label{prop:ord-hens-comp}
		Let $(K,<,v)$ be an ordered henselian valued field.
		Then any $\Lor$-definable valuation $w$ on $K$ is comparable to $v$.
	\end{proposition}
	\begin{proof}
		We may suppose that neither $v$ nor $w$ are trivial.
		We first claim that the valuation ring $\OO_w$ contains a set $U \neq \emptyset$ which is open in the topology induced by $v$.
		This follows from a suitable form of relative quantifier elimination for henselian valued fields of residue characteristic zero:
		In the terminology of \cite{hensel-min}, the $\Lvf$-theory of the valued field $(K,v)$ is $\omega$-h-minimal \cite[Corollary 6.2.6~(1.)]{hensel-min}.
		Let $P$ be the unary predicate on $\operatorname{RV} = K^\times/(1+\MM_v)$ given by $P(a(1+\MM_v))$ if and only if $a>0$ for any $a\in K^\times$. Since  elements of $1+\MM_v$ are squares and hence automatically positive, the positive cone of $(K,<)$ consists of all $a\in K^\times$ satisfying $P(a(1+\MM_v))$.
		As by \cite[Theorem 4.1.19]{hensel-min} $\omega$-h-minimality is preserved under expansions by additional predicates on $\operatorname{RV}$, we obtain that the $\Lovf$-theory of $(K,<,v)$ is $\omega$-h-minimal.
		In models of $\omega$-h-minimal theories, any infinite definable set contains a non-empty $v$-open ball  \cite[Lemma 2.5.2]{hensel-min}, proving our claim about $\OO_w$.
		
		It follows that $w$ cannot be independent from $v$, since otherwise Weak Approximation \cite[Theorem 2.4.1]{englerprestel} would imply that the set $U \cap (K \setminus \OO_w)$ is non-empty as the intersection of a $v$-open and a $w$-open set.
		
		Let us now suppose for a contradiction that $w$ and $v$ are incomparable.
		Let $v_0$ be the finest common coarsening of $v$ and $w$, and $\overline v$, $\overline w$ the induced valuations on the residue field $Kv_0$, which are non-trivial and independent.
		Writing $<$ for the induced ordering on $Kv_0$, we now have an ordered henselian valued field $(Kv_0, <, \overline v)$ with a valuation $\overline w$ independent from $\ol{v}$.
		By \Autoref{cor:stablemb}~\eqref{cor:stablemb:2}, $\overline w$ is $\Lor$-definable in $(Kv_0, <)$.
		Since $\overline w$ is independent from the henselian valuation $\overline v$ on $Kv_0$, this contradicts the first part of the proof.
	\end{proof}
	
	\section{Almost Real Closed Fields}\label{sec:arc}
	
	Following the terminology of \cite{delon}, we call a field $K$ \emph{almost real closed} if it admits a henselian valuation $v$ such that $Kv$ is real closed. Almost real closed fields arise in many valuation theoretic contexts, and they have been studied extensively (under varying  names) both algebraically and model-theoretically (cf., e.g., \cite{brown,becker,delon}). Due to the Baer--Krull Representation Theorem (cf.\ \cite[p.\ 37f.]{englerprestel}), any almost real closed field admits at least one ordering. In this section, we consider $\Lor$- and $\Lr$-definability of  valuations (which are \emph{a priori} not necessarily convex) in almost real closed fields. We establish in \Autoref{thm:arcmain} that \emph{any} $\Lor$-definable valuation is already $\Lr$-definable and henselian. Thereby we give a negative answer to \cite[Question~7.3]{krappkuhlmannlehericyforum}.

	Let $K$ be an almost real closed field. Then for any prime $p\in \N$ there exists a coarsest henselian valuation on $K$, denoted by $v_p$, with the property that $Kv_p=(Kv_p)^p\cup [-(Kv_p)^p]$ (cf.~\cite[page~1126~f.]{delon}). 	
	
	\begin{lemma}\thlabel{lem:psubgp}
		Let $p\in \N$ be prime and let $(K,v)$ be a henselian valued field with real closed residue field. Then $v(\OO_{v_p}^\times)$ is the maximal $p$-divisible convex subgroup of $vK$. 
	\end{lemma}
	\begin{proof}
		By \cite[Proposition~2.5~(iv)]{delon}, $v_p K = vK/v(\OO_{v_p}^\times)$ has no non-trivial $p$-divisible convex subgroup, so $v(\OO_{v_p}^\times)$ contains the maximal $p$-divisible convex subgroup of $vK$.
		On the other hand, $v$ induces a valuation on the residue field $Kv_p$ with value group $v(\OO_{v_p}^\times)$, from which it is easy to see that $v(\OO_{v_p}^\times)$ must itself be $p$-divisible by the defining property of $v_p$ (or see \cite[Lemma 2.4~(iii)]{delon}).
	\end{proof}

	With the results of the last section at our disposal, we can now imitate the proof of \cite[Theorem~4.4]{delon} to obtain the following.
	\begin{theorem}\thlabel{thm:arcmain}
		Let $K$ be an almost real closed field and let $<$ be any ordering on $K$. Then any $\Lor$-definable valuation on $(K,<)$ is henselian and $\Lr$-definable.
	\end{theorem}
	
	\begin{proof}
		Let us denote by $v_K$ the canonical henselian valuation on $K$.
		Since $K$ is almost real closed, $v_K$ coincides with the natural, i.e., finest convex valuation $\vnat$ on $(K,<)$ for any ordering $<$ on $K$, as $\vnat$ is henselian \cite[Proposition~2.1~(iv)]{delon}.
		
		Let $v$ be an $\Lor$-definable valuation on $K$.
		By \Autoref{prop:ord-hens-comp}, $v$ and  $v_K$ are comparable.
		If $v$ is a coarsening of $v_K$, then it is also henselian. Otherwise, $v_K$ is a strict coarsening of $v$.  Thus, by \Autoref{cor:stablemb}~\eqref{cor:stablemb:2} the non-trivial valuation that $v$ induces on $Kv_K$ is $\Lor$-definable in $(Kv_K,<)$, contradicting that $Kv_K$ is real closed. Hence, $v$ is henselian.
		
		In order to show that $v$ is $\Lr$-definable, by \cite[Theorem~4.4]{delon} it suffices to verify that $G_v:=v_K(\OO_v^\times)$ is $\Log$-definable in $v_KK$ and that $\OO_{v_p}\subseteq \OO_v$ for some prime $p\in \N$. The first condition follows from \Autoref{cor:stablemb}~\eqref{cor:stablemb:1}. For the other condition, we distinguish between two cases.
		
		\textbf{Case 1:} \emph{$v\neq v_K$.} Then $G_v\neq \{0\}$ and by \cite[Corollary~4.3]{delon} we have $G_p\leq G_v$ for some prime $p\in \N$, where $G_p$ denotes the maximal $p$-divisible convex subgroup of $v_KK$. Now $G_p=v_K(\OO_{v_p}^\times)$ by \Autoref{lem:psubgp}. Hence, $\OO_{v_p}\subseteq \OO_v$, as required. This establishes that $v$ is $\Lr$-definable in $K$.
		
		\textbf{Case 2:} \emph{$v=v_K$.} 
		Consider the set of formulas
		$$p(x)=\{x>n\wedge v(x)=0\mid n\in \N\}.$$
		This set is finitely satisfiable in $(K,<,v)$, i.e., a type. Hence, for some elementary extension
		$(L,<,v^*)$
		of $(K,<,v)$, there is some $x\in L$ with $v^*(x)=0$ and $x>n$ for all $n\in \N$. In particular, $v^*$ is a strict coarsening of the natural valuation $v_L$ on $(L,<)$. 
		By Case 1, $v^*$ is $\Lr$-definable in $L$. Thus, there exists an $\Lr$-formula $\varphi(x,\ul{y})$ such that $$(L,<,v^*)\models \exists \ul{y} \forall x \ (v^*(x)\geq 0 \leftrightarrow \varphi(x,\ul{y})).$$ By elementary equivalence, there exists $\ul{b}\in K$ such that $\varphi(x,\ul{b})$ defines $v$ in $K$.
	\end{proof}

	\section{Henselian Valuations}\label{sec:henselian}
	
	We now consider definability in general ordered henselian valued fields.
	Throughout this section, we freely use Farré's Ax--Kochen--Ershov Principles \cite[Corollary 4.2]{farre} (with all levels equal to $1$ in the notation there), stating that two ordered henselian valued fields are elementarily equivalent in $\Lovf$ if and only if the ordered residue fields and the value groups are so, and similarly for elementary extensions.
	
	Our first step is to show that a henselian valuation that is
	`slippery' in a precise sense involving residue field and value group cannot be $\Lor$-definable.
	\begin{lemma}\thlabel{lem:glitschig}
		Let $(K,<,v)$ be an ordered henselian valued field satisfying
		\begin{align*}(Kv,<)\equiv (L\pow{\Q},<) \text{ and }vK\equiv \Gamma \oplus \Q
		\end{align*}
		for some ordered field $L$ and some ordered abelian group $\Gamma$.
		Then $v$ is not $\Lor$-definable.
	\end{lemma}
	
	\begin{proof}
		Assume, for a contradiction, that $v$ were $\Lor$-definable. Fix an $\Lor$-formula $\varphi({x},\ul{y})$ such that for some $\ul{b}\in K$ the valuation ring $\OO_v$ is defined by $\varphi({x},\ul{b})$.		
		Since some instance of $\varphi(x, \ul{y})$ also defines the
		valuation $w$ in any
		$$(M,<,w) \equiv (K,<,v),$$
		we may assume that $(K,<,v)$ is in fact equal to $$(L((\mathbb{Q}))\underbrace{((\mathbb{Q}))((\Gamma))}_w, <, w),$$ where $w$ denotes the power series valuation with value group
		$\Gamma \oplus \mathbb{Q}$ ordered lexicographically.
		Let $v_\Gamma$ denote the power series valuation on $K$ with value group $\Gamma$ and residue field $L\pow{\mathbb{Q}}\pow{\mathbb{Q}}$.
		
		Applying \Autoref{cor:stablemb}~\eqref{cor:stablemb:2} to the ordered residue
		field of $(K,<,v_\Gamma)$, the $\Lor$-definability of
		$v$ implies that $\overline{v}$ (i.e., the valuation induced by $v$ on the residue field of its coarsening $v_\Gamma$) is also $\Lor$-definable on
		$(L((\mathbb{Q}))((\mathbb{Q})), <)$.
		
		Applying \Autoref{cor:stablemb}~\eqref{cor:stablemb:1} to the
		value group of $(L((\mathbb{Q}))((\mathbb{Q})), <, \overline{v})$,
		the convex subgroup $\Q$ corresponding to $\ol{v}$ is already
		$\Log$-definable in $\Q \oplus \Q$.
		This is a contradiction, as 
		$\Q \oplus \Q$ is divisible, and divisible ordered abelian groups
		admit no non-trivial proper $\Log$-definable subgroups.
	\end{proof}
	
	We now prove a lemma used to define coarsenings of a valuation that is
	essentially already shown in \cite{jahnkekoenigsmanncoarsenings}. It states
	that although in an ordered abelian group $G$, the smallest convex subgroup
	containing a given element $\gamma \in G$ need not be definable (e.g., the convex subgroup generated by $(0,1)$ in the lexicographic sum 
	$\mathbb{Q} \oplus \mathbb{Z}$ is not definable), it is definable up to $p$-divisible `noise'.
	\begin{lemma}\label{lem:pdivsub}
		Let $p \in \N$ be prime.
		There exists an $\Log$-formula $\varphi(x,y)$ such that the following holds:
		Let $G$ be an ordered abelian group and $\gamma \in G^{>0}$, and let $\langle \gamma \rangle$ denote the smallest convex subgroup
		of $G$ that contains $\gamma$.
		Then the set $\Delta_\gamma\subseteq G$ defined by $\varphi(x,\gamma)$ in $G$ is  
		the maximal convex subgroup of $G$ containing $\gamma$ such that $\Delta_\gamma/\langle \gamma \rangle$ is $p$-divisible.
	\end{lemma}
	
	\begin{proof}
		We set
		$\varphi(x,\gamma)$ to express
		$$[0,p|x|]\subseteq [0,p\gamma]+pG.$$
		By \cite[Lemma~4.1]{jahnkekoenigsmanncoarsenings}, $\Delta_\gamma$ is a convex subgroup of $G$ with $\gamma\in \Delta_\gamma$ such that no non-trivial convex subgroup of $G/\Delta_\gamma$ is $p$-divisible.
		In particular, for every convex subgroup $\Delta$ of $G$ properly containing $\Delta_\gamma$, the group  $\Delta/\langle\gamma\rangle$ is not $p$-divisible, since it has $\Delta/\Delta_\gamma \leq G/\Delta_\gamma$ as a quotient.
		
		On the other hand, every positive element $\delta \in \Delta_\gamma$ can by definition be written as the sum of an element of $[0, p\gamma] \subseteq \langle\gamma\rangle$ and an element of $pG$, which implies that $\Delta_\gamma/\langle\gamma\rangle$ is $p$-divisible.
	\end{proof}

	We extract the following consequence of the definability results of \cite{jahnkekoenigsmann1}.
	See the introduction of that paper for the notion of $p$-henselianity used in the proof.
	\begin{proposition}\label{prop:defble-refinement}
		Let $(K,v)$ be a henselian valued field such that the residue field $Kv$ is neither separably closed nor real closed.
		Then there exists an $\Lr$-definable (not necessarily henselian) refinement $w$ of $v$.
	\end{proposition}
	\begin{proof}

		Note that $Kv$ either has a Galois extension of degree divisible by some prime $p\neq 2$ or it only has Galois extensions of $2$-power degree. 
		
		In the latter case, since $Kv$ is neither separably closed nor real closed, it has Galois extensions of degree $2$ but is not Euclidean.
		Since $v$ is henselian and thus, in particular, $2$-henselian, we can now apply \cite[Corollary 3.3]{jahnkekoenigsmann1} to obtain that it admits an $\Lr$-definable refinement.
		Indeed, let $v_K^2$ denote the canonical $2$-henselian valuation (see \cite[page~743]{jahnkekoenigsmann1}).
		Then by \cite[Corollary 3.3]{jahnkekoenigsmann1}, $v_K^2$ is $\Lr$-definable if
		$Kv_K^2$ is non-Euclidean, otherwise the coarsest $2$-henselian valuation with Euclidean residue field $v_K^{2*}$ is $\Lr$-definable. Either is a refinement of $v$.
		
		Thus, we now assume that there is a prime $p\neq 2$ such that $Kv$
		has a finite Galois extension $L$ of degree divisible by $p$.
		Then there is a
		finite separable extension $M/Kv$ such that $L/M$ is a finite Galois
		extension of degree $p^n$ for some $n>0$ (e.g., take $M$ to be the fixed field of the $p$-Sylow subgroup of $\mathrm{Gal}(L/Kv)$ inside $L$).
		Let $F_0/K$ be a finite separable extension such that the (by henselianity unique) prolongation of $v$ to $F_0$ has residue field $M$ (see \cite[Theorem~5.2.7~(2)]{englerprestel} for the existence of $F_0$).
		Consider $F=F_0$ if the characteristic of $K$ is $p$, and $F = F_0(\zeta_p)$ otherwise,
		where $\zeta_p$ is a primitive $p$-th root of unity. 
		Then $F/K$ is a finite separable extension, and the residue field of the unique prolongation $u$ of $v$ to $F$ is a finite extension of $M$.
		In particular, $Fu$ admits Galois extensions of $p$-power degree, e.g., the compositum of $L$ and $Fu$.
		Therefore $u$ is a henselian (and thus in particular $p$-henselian) valuation with $Fu \neq Fu(p)$, and hence coarsens the canonical $p$-henselian valuation $v_F^p$ of $F$, which is $\emptyset$-$\Lr$-definable in $F$ by the Main Theorem of \cite{jahnkekoenigsmann1}.
		
		Now $F$ is interpretable in $K$ as the splitting field of a separable polynomial (cf.~\cite[page~31]{marker}).
		Hence $w=v_F^p|_K$ is $\Lr$-definable in $K$.
		Lastly, $v=u|_K$ is a coarsening of $w$, as $u$ is a coarsening of $v_F^p$ in $F$. 
	\end{proof}
	
	We can now state our main theorem about definability of henselian valuations on ordered fields.
	In the proof, we will need the following notion:
	for $n\in \N$, we say that a valuation $v$ on a field $K$ is \emph{$n_{\leq}$-henselian} if Hensel's Lemma holds in $(K,v)$ for all polynomials of degree at most $n$.
	Note that for a fixed $n$, the property of $n_{\leq}$-henselianity is elementary in the language $\Lvf$.
	
	\begin{theorem}\thlabel{thm:main}
		Let $(K,<,v)$ be an ordered henselian valued field. If $v$ is $\Lor$-definable, then it is $\Lr$-definable.
	\end{theorem}
	
	\begin{proof}
		If $v$ is trivial, then the proof is clear, thus we assume that $v$ is non-trivial from now on. 
		If $K$ is almost real closed, then the result follows from \Autoref{thm:arcmain}, hence we also assume that $K$ is not almost real closed.
		Let $v_K$ denote the canonical henselian valuation on $K$, i.e., the finest henselian valuation on $K$. In particular, $v_K$ is a (not necessarily proper) refinement of $v$.
		The residue field $Kv_K$ carries an ordering induced by the ordering on $K$, but is not real closed since $K$ is not almost real closed by assumption.
		By \Autoref{prop:defble-refinement}, we may thus fix an $\Lr$-definable refinement $w$ of $v_K$ (and hence of $v$).
		
		\smallskip
		Now, by \Autoref{lem:glitschig}, 
		we can make the following case distinction.
		
		\textbf{Case 1:} \emph{$vK \not\equiv \Gamma \oplus \Q$ for any ordered abelian group $\Gamma$.}

		Let $\Delta_v\leq wK$ be the convex subgroup such that $wK/\Delta_v=vK$. 
		We first show that there is a prime $p$ such that $vK$ contains no non-trivial $p$-divisible convex subgroup.
		Assume for a contradiction that $vK$ contains a non-trivial $p$-divisible convex subgroup for every prime $p$.
		Then any sufficiently saturated elementary extension $G^*$ of $vK$ contains a non-trivial convex divisible subgroup $Q$.
		Now, \cite[Lemma~1.11]{schmitt} implies 
		$$vK\equiv G^*\equiv G^*/Q\oplus Q\equiv G^*/Q\oplus \Q,$$ 
		contradicting that we are in Case 1.
		
		Hence, we can fix some prime $p$ such that $vK$ does not contain any non-trivial $p$-divisible convex subgroup.
		By \Autoref{lem:pdivsub}, there exists an $\Log$-formula $\varphi(x,y)$ such that for any positive $\gamma \in wK$, the subgroup $\Delta_\gamma$ of $wK$ defined by $\varphi(x,\gamma)$ is the maximal convex subgroup of $wK$
		containing $\langle \gamma \rangle$ and such that 
		$\Delta_\gamma/\langle \gamma \rangle$ is $p$-divisible.
		In case we choose $\gamma \in \Delta_v$, we have
		$\Delta_\gamma\leq \Delta_v$: otherwise
		$$ \langle \gamma \rangle \leq \Delta_v \lneq \Delta_\gamma
		$$
		implies that $\Delta_\gamma/\Delta_v \leq vK$ is a non-trivial
		convex subgroup which is $p$-divisible since it is a quotient of
		the $p$-divisible group $\Delta_\gamma/\langle \gamma \rangle$.
		
		For every $\gamma\in wK$, let $u_\gamma$ be the $\Lr$-definable coarsening of $w$ on $K$ with value group $wK/\Delta_\gamma$.
		Since $\Delta_\gamma$ is uniformly $\Log$-definable in $wK$, also $u_\gamma$ is uniformly $\Lr$-definable in $K$, i.e., there exists an $\Lr$-formula $\psi(x,\ul{y},z)$ and a parameter tuple $\ul{b}\in K$ such that for every $a\in K^\times$ the formula $\psi(x,\ul{b},a)$ defines $u_{w(a)}$.
		
		If $u_\gamma$ is already henselian for some $\gamma\in \Delta_v$, then $v$ is an $\Lor$-definable coarsening of an $\Lr$-definable henselian valuation and hence $v$ is $\Lr$-definable by \Autoref{cor:coarseninglr}. Thus, we assume that for every $\gamma\in \Delta_v$ the valuation $u_\gamma$ is not henselian.
		
		First suppose that there is some $n\in \N$ such that for every $\gamma\in \Delta_v$ we have that $u_\gamma$ is not $n_\leq$-henselian. Let $B$ be the $\Lr$-definable subset of $K$ consisting of all $a\in K^\times$ such that $u_{w(a)}$ is not $n_\leq$-henselian.
		We claim that $w(B)=\Delta_v$ holds.
		Let $a\in K^\times$ and set $\gamma=w(a)$. First suppose that $\gamma \in \Delta_v$. Then $u_\gamma=u_{w(a)}$ is not $n_\leq$-henselian. Thus, $a\in B$ and $\gamma\in w(B)$. Conversely, suppose that $\gamma \notin \Delta_v$. Then $\Delta_v \lneq \Delta_\gamma$ and thus $u_\gamma$ is a strict coarsening of $v$. Since $v$ is henselian, $u_\gamma$ is $n_\leq$-henselian. Hence, $a\notin B$ and $\gamma\notin w(B)$, as required.
		Thus, in this case $v$ is $\Lr$-definable as 
		$\OO_v$ consists exactly of all $x\in K$ with $w(x)\geq 0\vee w(x)\in w(B)$, which is an $\Lr$-definable condition as $w$ is $\Lr$-definable.
		
		Now suppose that for every $n\in \N$ there exists $\gamma_n\in \Delta_v$ such that $u_{\gamma_n}$ is $n_\leq$-henselian. 
		Then for every $n\in \N$, there is some $a_n\in K$ (with $w(a_n)=\gamma_n$) such that  $\psi(x,\ul{b},a_n)$ defines an $n_\leq$-henselian refinement of $v$ in $(K,<,v)$. 
		Hence, in some sufficiently saturated elementary extension $(K^*,<,v^*)$ of $(K,<,v)$, there exists $a\in K^*$ such that $\psi(x,\ul{b},a)$ defines a henselian refinement $u^*$ of $v^*$. Since $v^*$ is $\Lor$-definable by the same formula in $K^*$ as $v$ in $K$, it is an $\Lor$-definable coarsening of the $\Lr$-definable henselian valuation $u^*$. Hence, $v^*$ is $\Lr$-definable in $K^*$ by \Autoref{cor:coarseninglr} and thus also $v$ is $\Lr$-definable in $K$.
		
		\smallskip
		\textbf{Case 2:} 
		$(Kv,<)\not\equiv (L\pow{\Q},<)$ for any ordered field $(L, <)$.

		First suppose that $v\neq v_K$. Then $v_K$ is strictly finer than $v$. 
		Assume that $\overline{v_K}(Kv)$ is divisible,
		where $\overline{v_K}$ denotes the valuation induced by $v_K$ on $Kv$. 
		Then $\overline{v_K}(Kv)$ is elementarily equivalent to $\Q$ as an ordered abelian group, and thus
		we obtain
		\begin{align*}
			(Kv,<)&\equiv (Kv_K\pow{\ol{v_K}(Kv)},<)\equiv(Kv_K\pow{\Q},<),
		\end{align*}
		in contradiction to the assumption of Case 2.
		
		Therefore we can assume that $\overline{v_K}(Kv)=\Delta$ is non-divisible. We show that there exists some $\Lr$-definable henselian refinement $u_\gamma$ of $v$.
		
		We may assume that
		$\Delta$ does not have a rank $1$ quotient: otherwise we could consider a sufficiently saturated elementary extension $(K^*,<^*, v^*,w^*)$
		of $(K, <, v,w)$ in which --
		by the definability of the refinement $w$ of $v_K$ -- $\overline{w^*}(K^*v^*)$ (and hence also $\overline{v_{K^*}}(K^*v^*)$)
		has no rank $1$ quotient. Just as in Case 1, an $\Lr$-definition of $v^*$
		would give rise to an $\Lr$-definition of $v$.
		We claim that there is a prime $p$ such that $\Delta$ has no non-trivial $p$-divisible quotient.
		If not, then some saturated elementary extension $\Delta^*$ of $\Delta$ has a divisible non-trivial quotient $\Delta^*/\Gamma$, where $\Gamma$ is a convex proper subgroup.
		Then, as before, we have 
		\begin{align*}(Kv, <) &\equiv (Kv_K\pow{\ol{v_K}(Kv)},<) = (Kv_K\pow{\Delta}, <) \\&\equiv (Kv_K\pow{\Delta^*}, <)\equiv (Kv_K\pow{\Gamma}\pow{\Delta^*/\Gamma}, <) \equiv (Kv_K\pow{\Gamma}\pow{\Q}, <),
		\end{align*}
		contradicting that we are in Case 2. 
		Hence, there is a prime $p$ such that $\Delta$ is $p$-antiregular, i.e.,
		$\Delta$ has no non-trivial $p$-divisible quotient and no rank $1$ quotient 
		(cf.~\cite[page~670]{jahnkekoenigsmanncoarsenings}). 
		
		Recall that $w$ is an $\Lr$-definable refinement of $v_K$. Then, there
		are convex subgroups $$\Delta_{v_K} \lneq \Delta_v \leq wK$$ with
		$v_KK=wK/\Delta_{v_K}$ and $vK=wK/\Delta_v$. 
		
		Let $\gamma \in \Delta_v \setminus \Delta_{v_K}$ be positive, and let $\langle \gamma \rangle \leq wK$ denote the smallest convex subgroup containing $\gamma$. Since the convex subgroups of $wK$
		are ordered by inclusion, we have
		$$\Delta_{v_K} \lneq \langle \gamma \rangle \leq \Delta_v.$$
		
		Note that $\langle \gamma \rangle$ need not be $\Log$-definable in $wK$.
		However, by \Autoref{lem:pdivsub}, the maximal convex subgroup $\Delta_{\gamma} \leq wK$ that contains $\langle \gamma \rangle$ and such that $\Delta_{\gamma}/ \langle \gamma \rangle$
		is $p$-divisible
		is $\Log$-definable in $wK$. 
		
		We claim that $\Delta_\gamma \leq \Delta_v$, i.e.,
		that $\Delta_\gamma$ corresponds to an $\Lr$-definable refinement
		of $v$. Assume for a contradiction that we have $\Delta_v \lneq
		\Delta_\gamma$. Since $\Delta_v$ contains $\langle \gamma \rangle$, the choice of $\Delta_\gamma$ implies that $\Delta_v/\langle \gamma \rangle$
		is $p$-divisible. If $\langle \gamma \rangle \neq \Delta_v$, then $\Delta_v/ \langle \gamma \rangle$ is a non-trivial
		$p$-divisible quotient of $\Delta = \Delta_v/\Delta_{v_K}$, a contradiction. Otherwise, $\langle \gamma \rangle = \Delta_v$ and the quotient of $\langle \gamma \rangle$ by its maximal convex subgroup not containing $\gamma$ is of rank $1$, also a contradiction.
		Thus, we have found an $\Lr$-definable refinement of $v$. Since
		we have $$\Delta_{v_K} \lneq \Delta_\gamma,$$
		this refinement is a coarsening of $v_K$ and thus henselian.

		Now suppose that $v=v_K$. If $Kv$ is not $t$-henselian, then $v$ is $\Lr$-definable by \cite[Proposition~5.5]{fehmjahnke}. Hence, suppose that $Kv$ is $t$-henselian.
		Then for a sufficiently saturated elementary extension $(K^\ast, <^\ast, v^\ast) \succeq (K,<,v)$ the residue field $K^\ast v^\ast \succeq Kv$ is itself henselian.
		Since $v$ is $\Lor$-definable in $K$, also $v^\ast$ is $\Lor$-definable in $K^\ast$.
		However, since $K^\ast v^\ast$ is henselian, $v^\ast$ is not the canonical henselian valuation of $K^\ast$, and therefore by the arguments above $v^\ast$ is already $\Lr$-definable in $K^\ast$.
		Hence, also $v$ is $\Lr$-definable in $K$.
	\end{proof}


\begin{thebibliography}{99}

		\bibitem{aschenbrenner}
		\textsc{M.~Aschenbrenner}, \textsc{L.~van den Dries} and \textsc{J.~van der Hoeven},
		\textsl{Asymptotic Differential Algebra and Model Theory of Transseries}, Ann.\ of Math.\ Stud.\ 195 (Princeton Univ.\ Press, Princeton, NJ, 2017),
		doi:10.1515/9781400885411.
		
		\bibitem{becker}
		\textsc{E.~Becker}, \textsc{R.~Berr} and \textsc{D.~Gondard},
		`Valuation Fans and Residually Real Closed Henselian Fields',
		\textsl{J.\ Algebra} 215 (1999) 574--602,
		doi:10.1006/jabr.1998.7708.
		
		\bibitem{brown}
		\textsc{R.~Brown},
		`Automorphisms and isomorphisms of real Henselian fields',
		\textsl{Trans.\ Amer.\ Math.\ Soc.} 307 (1988) 675--703,
		doi:10.2307/2001193.
		
		\bibitem{hensel-min}
		\textsc{R.~Cluckers}, \textsc{I.~Halupczok} and \textsc{S.~Rideau-Kikuchi},
		`Hensel minimality I',
		\textsl{Forum Math. Pi} 10:e11 (2022) 1--68,
		doi:10.1017/fmp.2022.6.
		
		\bibitem{delon} 
		\textsc{F.~Delon} and \textsc{R.~Farré}, 
		`Some model theory for almost real closed fields', 
		\textsl{J.\ Symb.\ Log.} 61 (1996) 1121--1152, 
		doi:10.2307/2275808.
		
		\bibitem{dries}
		\textsc{L.~van~den~Dries},
		`Lectures on the Model Theory of Valued Fields',
		\textsl{Model Theory in Algebra, Analysis and Arithmetic}, Lect.\ Notes Math.\ 2111 (eds D.~Macpherson and C.~Toffalori; Springer, Heidelberg, 2014) 55--157, doi:10.1007/978-3-642-54936-6\_4.
		
		\bibitem{englerprestel} 
		\textsc{A.~J.~Engler} and \textsc{A.~Prestel},  
		\textsl{Valued Fields}, 
		Springer Monogr. Math. (Springer, Berlin, 2005).
		
		\bibitem{farre} 
		\textsc{R.~Farré}, 
		`A transfer theorem for Henselian valued and ordered fields', 
		\textsl{J. Symb. Log.} 58 (1993) 915--930, 
		doi:10.2307/2275104.	
		
		\bibitem{fehmjahnke} 
		\textsc{A.~Fehm} and \textsc{F.~Jahnke},
		`On the quantifier complexity of definable canonical Henselian valuations', 
		\textsl{Math. Log. Q.} 61 (2015) 347--361,
		doi:10.1002/malq.201400108.	
		
		\bibitem{fehmjahnkesurvey} 
		\textsc{A.~Fehm} and \textsc{F.~Jahnke}, 
		`Recent progress on definability of Henselian valuations', 
		\textsl{Ordered Algebraic Structures and Related Topics}, Contemp. Math. 697 (eds F.~Broglia, F.~Delon, M.~Dickmann, D.~Gondard-Cozette and V.~A.~Powers; Amer.
		Math. Soc., Providence, RI, 2017), 135--143,
		doi:10.1090/conm/697/14049.
		
		\bibitem{hongthesis} 
		\textsc{J.~Hong}, 
		`Immediate expansions by valuations of fields', 
		Doctoral Thesis, McMaster University, 2013.
		
		\bibitem{hong} 
		\textsc{J.~Hong}, 
		`Definable non-divisible Henselian valuations', 
		\textsl{Bull. Lond. Math. Soc.} 46 (2014) 14--18, 
		doi:10.1112/blms/bdt074.			

		\bibitem{jahnkekoenigsmann1} 
		\textsc{F.~Jahnke} and \textsc{J.~Koenigsmann}, 
		`Uniformly defining $p$-henselian valuations', 
		\textsl{Ann. Pure Appl. Logic} 166 (2015) 741--754, 
		doi:10.1016/j.apal.2015.03.003.
		
		\bibitem{jahnkekoenigsmanncoarsenings} 
		\textsc{F.~Jahnke} and \textsc{J.~Koenigsmann}, 
		`Defining coarsenings of valuations', 
		\textsl{Proc. Edinb. Math. Soc. (2)} 60 (2017) 665--687,
		doi:10.1017/S0013091516000341.	
		
		\bibitem{jahnkesimonwalsberg} 
		\textsc{F.~Jahnke}, \textsc{P.~Simon} and \textsc{E.~Walsberg}, 
		`Dp-minimal valued fields', 
		\textsl{J. Symb. Log.} 82 (2017) 151--165, 
		doi:10.1017/jsl.2016.15.
		
		\bibitem{jahnkesimon}
		\textsc{F.~Jahnke} and \textsc{P.~Simon},
		`NIP henselian valued fields', 
		\textsl{Arch. Math. Logic} 59 (2020) 167--178, 
		doi:10.1007/s00153-019-00685-8.	
		
		\bibitem{johnson}
		\textsc{W.~Johnson}, 
		`Dp-finite fields VI: the dp-finite Shelah conjecture', Preprint, 2020,
		arXiv:2005.13989.
		
		\bibitem{krappkuhlmannlehericyforum} 
		\textsc{L.~S.~Krapp}, \textsc{S.~Kuhlmann} and \textsc{G.~Lehéricy},
		`Ordered fields dense in their real closure and definable convex valuations', \textsl{Forum Math.} 33 (2021) 953--972,  
		doi:10.1515/forum-2020-0030. 
		
		\bibitem{krappkuhlmannlink} 
		\textsc{L.~S.~Krapp}, \textsc{S.~Kuhlmann} and \textsc{M.~Link},
		`Definability of henselian valuations by conditions on the value group', to appear in \textsl{J.\ Symb.\ Log.}, 2022,
		doi:10.1017/jsl.2022.34.
		
		\bibitem{malcev}
		\textsc{A.~I.~Mal'cev},
		`On the undecidability of elementary theories of certain fields', \textsl{Sibirsk.\ Mat.\ Ž.} 1 (1960) 71--77 (Russian),
		\textsl{Amer.\ Math.\ Soc.\ Translat., II.\ Ser.} 48 (1965) 36--43 (English),
		doi:10.1090/trans2/048.
		
		\bibitem{marker} 
		\textsc{D.~Marker},
		\textsl{Model Theory: An Introduction}, 
		Grad.\ Texts in Math.\ 217 (Springer, New York, 2002), 
		doi:10.1007/b98860.
		
		\bibitem{schmitt}
		\textsc{P.~H.~Schmitt},
		`Model theory of ordered abelian groups',
		Habilitationsschrift, Ruprecht-Karls-Universität Heidelberg, 1982.
		
		\bibitem{viswanathan} 
		\textsc{T.~M.~Viswanathan}, 
		`Ordered fields and sign-changing polynomials', 
		\textsl{J. Reine Angew. Math.} 296 (1977) 1--9, 
		doi:10.1515/crll.1977.296.1.
		
	\end{thebibliography}
\end{document}